\documentclass[11pt, reqno]{amsart}
\usepackage[utf8x]{inputenc}
\usepackage{amsmath, amsthm, amscd, comment, enumerate, textcomp, units,stmaryrd, wasysym}
\usepackage{amssymb, cite, datetime, extarrows}
\usepackage[all, arrow]{xy}\SelectTips{cm}{10}
\makeatletter
\def\@setcopyright{}
\def\serieslogo@{} 
\makeatother 

\DeclareMathOperator{\rank}{\mathrm{rank}}

\newcommand{\cyc}{\textup{cyc}}
 
\newcommand{\di}{\textup{div}}
\newcommand{\ab}{\textup{ab}}

\newcommand{\gm}{\GG_m(\QQbar)}
\newcommand{\GQ}{\mcG(\QQbar)}
\newcommand{\GK}{\mcG(k)}
\newcommand{\GQM}{\GQ/\GQ_\tors}
\title[Remarks]{Remarks on R\'{e}mond's generalized Lehmer problems\footnote{\today}\hfill}
\author{Robert Grizzard}
\subjclass[2010]{11G50, 11J25, 11R04, 11R06}
\keywords{Weil height, Lehmer problem}
\address{Department of Mathematics, Lafayette College, Easton, Pennsylvania 18042 USA}
\email{grizzarr@lafayette.edu}
\numberwithin{equation}{section}
\makeatletter
\def\imod#1{\allowbreak\mkern10mu({\operator@font mod}\,\,#1)}
\makeatother

\newcommand{\Gal}[0]{\operatorname{Gal}}

 \newtheorem{theorem}{Theorem}[section]

 \newtheorem{corollary}[theorem]{Corollary}
   
 \newtheorem{conjecture}[theorem]{Conjecture}

 \newtheorem*{claim*}{Claim}

 \newcommand{\mc}{\mathcal}

 \newcommand{\mcG}{\mc{G}}

  \newcommand{\GG}{\mathbb{G}}
 \newcommand{\NN}{\mathbb{N}}
 
 \newcommand{\QQ}{\mathbb{Q}}

 \newcommand{\ZZ}{\mathbb{Z}}
 
 \newcommand{\tors}[0]{\textup{tors}}
 \newcommand{\QQbar}[0]{\overline{\QQ}}
 \newcommand{\vep}{\varepsilon}

 \newcommand{\ba}{\boldsymbol a}
 \newcommand{\bb}{\boldsymbol b}

\begin{document}

\bibliographystyle{alpha} 
\begin{abstract} 
 We draw connections between the various conjectures which are included in G. R\'{e}mond's ``generalized Lehmer problems.''  Specifically, we show that the degree one form of his conjecture for $\GG_m$ is, in a sense, almost as strong as the strong form.  Our insights into the conjectures apply a basic result on what may be called Diophantine approximation in the metric induced by the canonical height.  We also present a previously unpublished partial result of R\'{e}mond on his conjectures.  
\end{abstract}

\maketitle 


\section{Introduction}

Throughout this paper we work in a fixed algebraic closure $\QQbar$ of $\QQ$.  
Let $\mcG$ be either a torus $\GG_m^r$ or an abelian variety over $\QQ$, equipped with a height function $h:\mcG(\QQbar) \to [0,\infty)$ satisfying the properties
\begin{enumerate}
\item[(1)] $h(P) = 0$ if and only if $P$ is a torsion point,
\item[(2)] $h([n] P) = |n|h(P)$ for any $n \in \ZZ$, $P \in \mcG(\QQbar)$,
\item[(3)] $h(P+Q) \leq h(P)+h(Q)$ for any $P,Q \in \mcG(\QQbar)$, and 
\item[(4)] $h(P^\sigma) = h(P)$ for any $P \in \GQ$ and any $\sigma \in \Gal(\QQbar/\QQ)$.
\end{enumerate}
For example, if $\mcG = \GG_m$, we can take $h$ to be the usual logarithmic absolute Weil height (hereafter simply ``Weil height''), and it satisfies these properties -- see for example \cite[Section 1.5]{bombierigubler}.
Of course, in this setting, the group law is multiplication, so the properties look different; e.g. (3) becomes: $h(\alpha \beta) \leq h(\alpha)+h(\beta)$ for all $\alpha, \beta \in \QQbar^\times$, and so on.   If $\mcG = \GG_m^r$ is a torus, we take the sum of the Weil heights in each coordinate.  If $\mcG$ is an elliptic curve, or more generally an abelian variety, we may take $h$ to be the square root of a N\'{e}ron-Tate canonical height $\hat{h}$ (see \cite[Chapter 9]{bombierigubler}).  

The quotient group $\GQM,$ being both divisible and torsion-free, is a vector space over $\QQ$.  It is easy to show using properties (1) through (3) that $h$ is well defined on this quotient, and is in fact a norm with respect to the usual absolute value on $\QQ$.

If $\Gamma$ is any subgroup of $\GQ,$ we define a height function $h_\Gamma :  \GQ \to [0,\infty)$ relative to $\Gamma$ as follows:
\begin{equation}\label{hgamma}
h_\Gamma(P) = \inf_{Q \in \Gamma} h(P-Q) = \inf_{Q \in \Gamma} h(P+Q).
\end{equation}
Such functions have been explored previously in \cite{dlmf} in the case where $\mcG=\GG_m$.  If we interpret $h(P-Q)$ as the distance\footnote{Of course, without taking the quotient by the torsion subgroup, this is not a true distance but a semidistance.} between $P$ and $Q$, then $h_\Gamma(P)$ is the distance from $P$ to the subgroup $\Gamma$.  Of course, taking $\Gamma$ to be the trivial group, or more generally any subgroup of torsion elements, reproduces the original height function $h$.
 
For any subfield $k$ of $\QQbar$, we write $G_k = \Gal(\QQbar/k)$, $k^\ab$ for the maximal abelian extension of $k$, and $k^\cyc$ for the (generally smaller) field obtained by adjoining to $k$ all roots of unity.  (Throughout we don't assume $k$ is a finite extension of $\QQ$ unless we say so explicitly.)
If $\Gamma$ is a subgroup of $\mcG(\QQbar)$, then we define $\Gamma^\di$ by 
\begin{equation}
\Gamma^\di = \{P \in \GQ ~\big|~ [n]P \in \Gamma ~\text{for some}~ n \in \NN\}.
\end{equation} 

Recall that the classical Lehmer problem, introduced by D. H. Lehmer in \cite{lehmer33} (in the case $\mcG = \GG_m$), asks whether there exists an absolute constant $c$ such that $h(P) \geq c/[\QQ(P):\QQ]$ for all $P \in \GQ$ not torsion.  The stronger ``relative Lehmer problem'' asks the same, but with $[\QQ(P):\QQ]$ replaced by $[\QQ^{\ab}(P):\QQ^{\ab}]$.  G. R\'{e}mond in \cite{remond} has introduced some interesting generalizations of these problems, and of their variants.  We summarize them here in the case of $\mcG = \GG_m$ for simplicity.  

\begin{conjecture}[R\'{e}mond]\label{gl}
Let $\Gamma = \Gamma^\di$ be a finite-rank\footnote{here the rank of $\Gamma$ is $\rank(\Gamma) = \dim_\QQ(\Gamma \otimes_\ZZ\QQ).$} subgroup of $\GG_m(\QQbar)$, i.e. $\Gamma = \langle g_1,\dots,g_n\rangle^\di$ for some multiplicatively independent elements $g_1,\dots,g_n$.  Let $K_\Gamma = \QQ(\Gamma)$ denote the smallest field of definition for the elements of $\Gamma$.
\begin{itemize}
\item \textbf{Strong form:} There exists a constant $C_\Gamma > 0 $ such that for all $\alpha \in \GG_m(\QQbar) \setminus \Gamma$ we have 
\begin{equation}\label{sf}
h(\alpha) \geq \frac{C_\Gamma}{[K_\Gamma(\alpha):K_\Gamma]}.
\end{equation}
\item \textbf{Weak form:} For any $\vep > 0$, there exists a constant $C_\Gamma(\vep) > 0 $ such that for all $\alpha \in \GG_m(\QQbar) \setminus \Gamma$ we have 
\begin{equation}\label{wf}
h(\alpha) \geq \frac{C_\Gamma(\vep)}{[K_\Gamma(\alpha):K_\Gamma]^{1+\vep}}.
\end{equation}
\item \textbf{Degree one form:} There exists a constant $C_\Gamma^0$ such that for all $\alpha \in \GG_m(K_\Gamma) \setminus \Gamma$ we have 
\begin{equation}\label{d1f}
h(\alpha) \geq C_\Gamma^0.
\end{equation} 
\end{itemize} 
\end{conjecture}   
 
We summarize how these conjectures relate to existing conjectures and results.  Dropping the condition $\Gamma = \Gamma^\di$ for a moment, if $\Gamma = \{1\}$, the strong form corresponds to Lehmer's problem, the weak form follows from Dobrowolski's theorem (see \cite{dobrowolski}; see also \cite[Section 4.4]{bombierigubler}), and the degree one form is trivial.  If $\Gamma$ is the group of all roots of unity (the only case where $\Gamma = \Gamma^\di$ has rank zero), the strong form becomes the relative Lehmer problem, the weak form follows from a theorem of Amoroso and Zannier (see \cite{amorosozannier2000}), and the degree one form is a result of Amoroso and Dvornicich (see \cite{amorosodvornicich2000}).  In all cases where the rank of $\Gamma$ is positive and finite, all forms of Conjecture \ref{gl} remain open in general.  For example, when $\Gamma = \langle2\rangle^\di,$ the field $K_\Gamma$ is generated by adjoining to the rationals all roots of binomials of the form $x^n-2$, for all $n \in \NN$.  Here the degree one form, which states that the only points of small height in this field are the roots of those binomials, is still open, although Amoroso has made some progress in \cite{amoroso}.
 


The connection to the heights $h_\Gamma = \inf_{\gamma \in \Gamma}h(\alpha/\gamma)$ we have discussed arises because all three forms of Conjecture \ref{gl} are equivalent to the identical statements with $h$ replaced by $h_\Gamma$.  This equivalence follows because, on the one hand, we clearly always have $h(\alpha) \geq h_\Gamma(\alpha)$; on the other hand, for any $\alpha \in \GQ$ and any $\beta \in \Gamma$, we have $K_\Gamma(\alpha\beta) = K_\Gamma(\alpha)$, and $\alpha \in \Gamma$ if and only if $\alpha\beta \in \Gamma$.  In light of this, we highlight the following conjecture for $h_\Gamma$, which is a weaker version of the weak form of Conjecture \ref{gl}.  

\begin{conjecture}\label{rem2}
If $\Gamma = \langle g_1,\dots, g_n\rangle^\di$ is a finite-rank, divisible subgroup of $\GG_m(\QQbar)$, then for any $\vep > 0$ there exists a constant $C_\Gamma(\vep) > 0$ such that for any $\alpha \in \GG_m(\QQbar) \setminus \Gamma$ we have 
\begin{equation}\label{hg2}
h_\Gamma(\alpha) \geq \frac{C_\Gamma(\vep)}{[k^{\cyc}(\alpha):k^{\cyc}]^{2 + \vep}},
\end{equation} 
where $k$ is the number field $\QQ(g_1,\dots,g_n)$.
\end{conjecture}

In this direction R\'{e}mond announced the following result in a 2013 talk\footnote{R\'{e}mond's talk, \emph{Generalized Lehmer Problems}, was given on November 25, 2013 as a part of the workshop ``Heights in Diophantine geometry, group theory and additive combinatorics'' held at the Erwin Schr\"{o}dinger Institute in Vienna.}.  We reproduce it here with permission; the proof is given in the next section.

\begin{theorem}[R\'{e}mond]\label{doubleweak}
If $\Gamma = \Gamma^\di$ is a finite-rank subgroup of $\GG_m(\QQbar)$, then for any $\vep > 0$ there exists a constant $C_\Gamma(\vep) > 0$ such that for any $\alpha \in \GG_m(\QQbar) \setminus \Gamma$ we have 
\begin{equation}\label{hg1ge}
h_\Gamma(\alpha) \geq \frac{C_\Gamma(\vep)}{[\QQ(\alpha):\QQ]^{\rank(\Gamma) + 1 + \vep}}.
\end{equation}
\end{theorem}


The following result, which is the main original observation of this paper, shows that the degree one form of Conjecture \ref{gl} implies that we can strengthen (\ref{hg1ge}) to (\ref{hg2}), and gives partial progress toward the strong form of Conjecture \ref{gl}.

\begin{theorem}\label{d1iskey}
For a finite-rank subgroup $\Gamma=\Gamma^\di$ of  $\GG_m(\QQbar)$, the degree one form of Conjecture \ref{gl} implies Conjecture \ref{rem2}.  

More precisely, let $\Gamma = \langle g_1,\dots,g_n\rangle^\di,$ and assume the degree one form of Conjecture \ref{gl} is true for the group $\Gamma$, so there is a constant $C_\Gamma^0$ such that (\ref{d1f}) holds.  Set $k = \QQ(g_1,\dots,g_n)$, and let $\alpha \in \GG_m(\QQbar) \setminus \Gamma.$  Then there exist constants $C_\Gamma(\vep)$ for all $\vep> 0$ such that
\begin{enumerate}
\item[(i)] if $\alpha \not \in (k^\times)^\di$, then for each $\vep >0$ we have
\begin{equation}\label{part1}
h_\Gamma(\alpha) \geq \frac{C_\Gamma(\vep)}{[k^\ab(\alpha):k^\ab]^{2 + \vep}};
\end{equation}
\item[(ii)] if $\alpha \in (k^\times)^\di$, then we have
\begin{equation}\label{part2}
h(\alpha) \geq h_\Gamma(\alpha) \geq \frac{C_\Gamma^0}{[K_\Gamma(\alpha):K_\Gamma]}.
\end{equation}
\end{enumerate}
\end{theorem}
Of course the inequality $h(\alpha) \geq h_\Gamma(\alpha)$ is trivial by definition, but is included to emphasize that the bound we achieve in case (i) is the one expected by the strong form of Conjecture \ref{gl}.  This is stronger than the conclusion of Conjecture \ref{rem2} of course, because $k^\cyc$ is contained in both $k^\ab$ and $K_\Gamma$.  One may infer from this result that, in a sense, the degree one part of Conjecture \ref{gl} may be as deep as the full conjecture.

The rest of the paper is organized as follows.  Section \ref{remproofsec} contains the proof of Theorem \ref{doubleweak}.  Section \ref{ineqsec} introduces a simple inequality (Theorem \ref{thma}) relating to $h_\Gamma$, which is essentially the same as \cite[Theorem 1.1]{gv}.  In that paper attention is restricted to $\mcG = \GG_m$; we reproduce the proof here in the more general case for convenience.  In Section \ref{nfsec} we illustrate how to make Theorem \ref{thma} effective when the field in question is a number field by applying ``Dobrowolski type'' estimates, and in the final section we use these effective versions to prove Theorem \ref{d1iskey}.


\subsection*{Acknowledgments}
The author wishes to thank Ga\"{e}l R\'{e}mond for permission to reproduce his results here, as well as for very useful conversations and notes on an early draft of this paper, and Joseph Gunther, who also provided notes on an early draft. 

\section{Proof of R\'{e}mond's theorem (Theorem \ref{doubleweak})}\label{remproofsec}
The following proof was sketched in the aforementioned 2013 talk of R\'{e}mond, and is based on the author's notes from that talk.
\begin{proof} 
Let $n = \rank(\Gamma)$ and $\Gamma = \langle g_1,\dots,g_n\rangle^\di$, where $g_1,\dots,g_n$ are multiplicatively independent, and set $k = \QQ(g_1,\dots,g_n).$  Fix $\alpha \in \GG_m(\QQbar) \setminus \Gamma$, and define a function $f:\QQ^n \to [0,\infty)$ by
\begin{equation}
f(\ba) = f(a_1,\dots,a_n) = h\left(\alpha g_1^{a_1}\cdots g_n^{a_n}\right),
\end{equation} 
so that $h_\Gamma(\alpha) = \inf_{\ba \in \QQ^n} f(\ba)$.  Notice that the basic properties of the height ensure that $f$ is well-defined, and satisfies a Lipschitz condition
\begin{equation}
|f(\ba)-f(\bb)| \leq L_\Gamma \|\ba-\bb\|_\infty,
\end{equation}
for any $\ba = (a_1,\dots,a_n)$ and $\bb = (b_1,\dots,b_n) \in \QQ^n$, where $L_\Gamma$ is a positive constant depending only on $\Gamma$, and $\|\ba-\bb\|_\infty := \max_{1 \leq i \leq n} |a_i-b_i|$.  Indeed, using the reverse triangle inequality for the height norm and other basic properties of the height, we have
\begin{align}
|f(\ba) - f(\bb)| &= |h(\alpha g_1^{a_1}\cdots g_n^{a_n}) - h(\alpha g_1^{b_1}\cdots g_n^{b_n})|  \leq  h\left(g_1^{a_1-b_1} \cdots g_n^{a_n-b_n}\right) \\
&\leq \sum_{i=1}^n |a_i-b_i|h(g_i) \leq \left(\sum_{i=1}^n h(g_i)\right)\|\ba-\bb\|_\infty. 
\end{align}

Using Dobrowolski's Theorem (inequality (\ref{dob})), we know that for any $\vep >0$ there is a constant $c(\vep) >0$ such that, if $\bb \in \frac{1}{m} \ZZ^n$, then
\begin{equation}\label{dob1}
f(\bb) = \frac{1}{m}h\left(\alpha^mg_1^{mb_1}\cdots g_n^{mb_n}\right) \geq \frac{c(\vep)}{m[k(\alpha):\QQ]^{1+\vep}}. 
\end{equation}

Fix an arbitrary $\ba \in \QQ^n$, and let $Q$ be any real number greater than 1.  By Dirichlet's approximation theorem (see for example \cite[Chapter VIII, Proposition 4]{coppel}, the case of an $n \times 1$ matrix), there exists a positive integer $m \leq Q^n$ and an element $\bb \in \frac{1}{m}\ZZ^n$ such that $\|\ba-\bb\|_\infty \leq \frac{1}{Qm}.$

Now we have
 
 \begin{align}
 f(\ba) &\geq f(\bb) - |f(\bb)-f(\ba)| \geq \frac{c(\vep)}{m[k(\alpha):\QQ]^{1+\vep}} - L_\Gamma \|\ba - \bb\|_\infty\\ &\geq \frac{c'(\Gamma, \vep)}{m[\QQ(\alpha):\QQ]^{1+\vep}} - \frac{L_\Gamma}{Qm} 
 \geq \frac{1}{Q^{n+1}} \left(\frac{Qc'(\Gamma,\vep)}{[\QQ(\alpha):\QQ]^{1+\vep}} -L_\Gamma \right) , 
 \end{align}
 where $c'(\Gamma,\vep) = c(\vep)/[k:\QQ]^{1+\vep}$.  Choosing $Q = (1+L_\Gamma)[\QQ(\alpha):\QQ]^{1+\vep}/c'(\Gamma,\vep)$, the above yields
 \begin{align}
f(\ba) &\geq \frac{1}{Q^{n+1}} = \frac{c'(\Gamma,\vep)^{n+1}}{(1+L_\Gamma)^{n+1}[\QQ(\alpha):\QQ]^{n+1+(n+1)\vep}},~\textup{and so}\\ 
f(\ba) &\geq \frac{c'(\Gamma,\vep/(n+1))^{n+1}}{(1+L_\Gamma)^{n+1}[\QQ(\alpha):\QQ]^{n+1+\vep}},
 \end{align}
 by replacing $\vep$ by $\vep/(n+1)$.
 This gives the desired result with
 \begin{equation}
 C_\Gamma(\vep) = \frac{c'(\Gamma,\vep/(n+1))^{n+1}}{(1+L_\Gamma)^{n+1}}. 
 \end{equation}
 \end{proof}


\section{A simple but important height inequality}\label{ineqsec}
For this section we return to the general case where $\mcG$ is a group equipped with a height function $h$ satisfying the properties listed at the beginning of the paper, and $k$ will denote an arbitrary subfield of $\QQbar$. For $P \in \mcG(\QQbar)$ and a subfield $k \subseteq \QQbar$ we define 
\begin{equation}
V_k(P) = h_{\GK^\di}(P), 
\end{equation}

and also we define
\begin{equation}\label{widthk}
W_k(P) = \max \left\{h(P^\sigma - P) ~\big|~ \sigma \in G_k\right\}.
\end{equation}

Notice that $W_k(P) \geq 0,$ and $W_k(P) = 0$ if and only if all conjugates $P^\sigma$ of $P$ over $k$ differ from $P$ by a torsion point, which is equivalent to saying that some multiple of $P$ lies in $\mcG(k)$, i.e. that $P \in \GK^\di$.  Indeed, if all Galois conjugates $P_1,\dots P_n$ of $P$ over $k$ differ from $P$ by torsion points, we must have that some multiple of the ``norm'' $N = \sum_i P_i \in \GK$ is equal to a multiple of $P$.

The following basic result relates $V_k$ and $W_k$.  It is the same result presented in \cite[Theorem 1.1]{gv}, stated there only for $\mcG = \GG_m.$
\begin{theorem}\label{thma}
Let $k$ be a subfield of $\QQbar$, and let $P \in \GQ$.  Then
\begin{equation}
W_k(P) \geq V_k(P) \geq \frac{1}{2}W_k(P).
\end{equation}
\end{theorem}

\begin{proof}

For any $P \in \GQ$ we have
\begin{equation}\label{hkinf}
V_k(P) = \inf_{Q \in \GK^\di} h(P-Q) = \inf_{\substack{Q \in \GK \\ m \in \NN}} h\left(P-\left[\frac{1}{m}\right]Q\right) = \inf_{\substack{Q \in \GK \\ m \in \NN}} \frac{1}{m}h\big([m]P-Q\big).
\end{equation}
Note that, using basic height properties, one can see that the quantity $h\left(P-\left[\frac{1}{m}\right]Q\right)$ above is independent of the choice of $\left[\frac{1}{m}\right] Q$, which denotes a element mapped to $Q$ by $[m]$. 

Let $\tau$ be an element of $G_{k}$ such that $W_{k}(P) = h(P^\tau - P)$.  Then for any $m \in \NN$ and $Q \in \GK$ (so $Q^\tau = Q$), we have
\begin{align}
W_{k}(P) &= \frac{1}{m} h\big(([m]P)^\tau -[m]P\big) = \frac{1}{m} h\big(([m]P-Q)^\tau - (Q- [m]P)\big)\\
&\leq \frac{1}{m} h\big( ([m]P-Q)^\tau \big) + \frac{1}{m} h\big(Q-[m]P\big) = \frac{2}{m}h \big([m]P-Q\big).
\end{align}
Taking the infimum as in (\ref{hkinf}), we find that $V_k(P) \geq \frac{1}{2}W_k(P).$

For the other inequality, let $P = P_1,\dots,P_n$ denote the Galois conjugates of $P$ over $k$.  Then $N := P_1+\cdots+P_n$ lies in $\GK,$ so
\begin{align}
V_k(P) \leq \frac{1}{n} h\big([n]P- N\big) \leq \frac{1}{n} \sum_{i=1}^n h(P - P_i) \leq \max_{1\leq i \leq n} h(P-P_i) = W_k(P).
\end{align} 
\end{proof}

Since $W_k(P) = 0$ if and only if $P \in \GK^\di$, an immediate consequence of the second inequality in Theorem \ref{thma} is the following.
\begin{corollary}\label{noapprox}
Let $k$ be a subfield of $\QQbar$, and let $P \in \GQ$.  Then $V_k(P) = 0$ if and only if $P \in \GK^\di$.
\end{corollary}
Corollary \ref{noapprox} states that, for any subfield $k$ of $\QQbar$, a point $P \in \GQ$ either lies in $\GK^\di$ or cannot be approximated (in the metric induced by the height) by such elements.  Note that \cite[Theorem 2]{dlmf} gives this result when $k$ is a number field and $\mcG = \GG_m.$  The general result for $\mcG = \GG_m$ is given in the preprint \cite{gv} (in that paper we generalize the result in a different direction).  The proof for our statement above with $\mcG$ a more general group is virtually identical.




\section{The number field case of Theorem \ref{thma} for $\GG_m$}\label{nfsec}
For the rest of the paper, we'll again restrict to the case $\mcG = \GG_m$, so for a subfield $k \subseteq \QQbar$, then, $\GK$ is simply the multiplicative group $k^\times$, and $h$ denotes the Weil height on nonzero algebraic numbers.  We'll write $k^\di$ as a shorthand for $(k^\times)^\di$.  Let $k$ be a number field, i.e. $[k:\QQ] < \infty$.  As discussed in the previous section, the inequality $V_k(\alpha) \geq \frac{1}{2}W_k(\alpha)$ implies that elements $\alpha$ which are not in $k^\di$ cannot be approximated by elements which are (Corollary \ref{noapprox}), and the ``gap'' is determined by the greatest distance from $\alpha$ to a conjugate, i.e. $W_k(\alpha)$.  By applying well-known height bounds of Dobrowolski type (these apply because $k$ is a number field) to estimate $W_k(\alpha)$, we find in this case a lower bound on $V_k(\alpha)$ in terms of the degree of $\alpha$.   

We briefly review some Dobrowolski type height estimates.  In the seminal work \cite{dobrowolski}, E. Dobrowolski showed that, for any non-torsion element $\alpha \in \QQbar^\times$ with $[\QQ(\alpha):\QQ] = d$, we have
\begin{equation}\label{dob}
h(\alpha) \geq \frac{c}{d}\left(\frac{\log \log 3d}{\log 3d}\right)^3
\end{equation}
for an absolute constant $c$.  The best known constant in a result of this type is the following bound:
\begin{equation}\label{vout}
h(\alpha) \geq \frac{1}{4d}\left(\frac{\log \log d}{\log d}\right)^3
\end{equation}
 valid for all $d > 2$, which was obtained by P. Voutier in \cite{voutier}.  Recall that the Lehmer conjecture is that the log factors can be removed from (\ref{dob}), and there is also the relative Lehmer conjecture, which states that we should have $h(\alpha) \geq c/D$, where $D = [k^{\ab}(\alpha):k^{\ab}]$ for some number field $k,$ and $c$ depends only on $k$.  In this direction F. Amoroso and U. Zannier have shown in \cite{amorosozannier2000} that
\begin{equation}\label{amza}
h(\alpha) \geq \frac{c(k)}{D}\left(\frac{\log \log 5D}{\log 2D}\right)^{13},
\end{equation}
where $c(k)$ is a positive constant depending only on the number field $k$.  This has been refined by Amoroso and E. Delsinne, who in \cite{amorosodelsinne} give the \emph{effective} lower bound
\begin{equation}\label{amdel}
h(\alpha) \geq \frac{1}{D_0}\frac{(\log \log 5D_0)^3}{(\log 2D_0)^4},
\end{equation}
where $D_0 = [\QQ^\ab(\alpha):\QQ^\ab];$ here we have taken the case $\mathbb{K} = \QQ$ of \cite[Th\'{e}or\`{e}me 1.3]{amorosodelsinne} for simplicity.

The following theorem applies such estimates to collect three of many possible lower bounds on $V_k(\alpha)$ in terms of the degree of $\alpha$.  We give two explicit lower bounds, and one without an explicit constant, but where the dependence is only on the degree $[k^\ab(\alpha):k^\ab]$.

\begin{theorem}\label{adob}
Let $k$ be a number field, and let $\alpha \in \gm$.  Let $d = [\QQ(\alpha):\QQ].$  If $\alpha \not \in k^\di$, then for $d >2$ we have
\begin{equation}\label{vout2}
V_k(\alpha) \geq \frac{1}{8d(d-1)} \left(\frac{\log \log d(d-1)}{\log d(d-1)}\right)^3.
\end{equation}  
Alternatively, writing $D_0 = [\QQ^{\ab}(\alpha):\QQ^{\ab}],$ we have for $D_0 >1$ that
\begin{equation}\label{amdel2}
V_k(\alpha) \geq \frac{1}{2D_0(D_0-1)} \frac{(\log \log 5D_0(D_0-1))^3}{(\log 2D_0(D_0-1))^4}. 
\end{equation}
Finally, writing $D = [k^\ab(\alpha):k^\ab]$, there is a constant $c(k)>0$ depending only on $k$ such that, for $D>1$, we have
\begin{equation}\label{amza2}
V_k(\alpha) \geq \frac{c(k)}{D(D-1)} \left(\frac{\log \log 5D(D-1)}{\log 2D(D-1)}\right)^{13}. 
\end{equation}
\end{theorem} 
 
\begin{proof}
For any $\sigma \in G_k$, notice that 
\begin{equation}\label{dd-1}
[\QQ(\sigma \alpha/\alpha):\QQ] \leq [\QQ(\alpha,\sigma \alpha):\QQ] \leq d(d-1),
\end{equation}
since $\sigma \alpha$ satisfies a polynomial of degree $d-1$ over $\QQ(\alpha)$.  Similarly we have
\begin{align}
[\QQ^{\ab}(\sigma \alpha/\alpha):\QQ^{\ab}] &\leq [\QQ^{\ab}(\alpha,\sigma \alpha): \QQ^{\ab}] \leq D_0(D_0-1), ~\textup{and}\label{D_0D_0-1}\\
[k^{\ab}(\sigma \alpha/\alpha):k^{\ab}] &\leq [k^{\ab}(\alpha,\sigma \alpha): k^{\ab}] \leq D(D-1).
\end{align} 

Assume $\alpha \not \in k^\di$, so that $h(\sigma \alpha/\alpha) >0$ for some $\sigma \in G_k$, and so $W_k(\alpha) > 0$.  Let $\tau$ be an element of $G_k$ such that $W_k(\alpha) = h(\tau \alpha /\alpha)$. 
Applying (\ref{vout}) and (\ref{dd-1}) we have
\begin{align}\label{long1}
W_k(\alpha) = h\left(\frac{\tau \alpha}{\alpha}\right) &\geq \frac{1}{4d(d-1)} \left(\frac{\log \log [\QQ(\tau\alpha/\alpha):\QQ]}{\log ([\QQ(\tau\alpha/\alpha):\QQ])}\right)^3\\ 
&\geq \frac{1}{4d(d-1)} \left(\frac{\log \log d(d-1)}{\log d(d-1)}\right)^3. 
\end{align}
If $\alpha \not \in k^\di$, then by Theorem \ref{thma} we have $V_k(\alpha) \geq \frac{1}{2}W_k(\alpha)$, which yields (\ref{vout2}).
The relative bounds (\ref{amdel2}) and (\ref{amza2}) follow in the same way by applying (\ref{amdel}) and (\ref{amza}), respectively.

\end{proof}


\section{Proof of Theorem \ref{d1iskey}}

\begin{proof}
In assuming the degree one form of Conjecture \ref{gl} we are furnished with a lower bound $C_\Gamma^0$ on the heights of points of $K_\Gamma\setminus \Gamma$.  
The first case we consider is when $\alpha \not \in k^\di$, where by the inequality (\ref{amza2}) of Theorem \ref{adob} we have for any $\vep > 0$ a constant $c(k,\vep)$ such that
\begin{equation}
h_\Gamma (\alpha) \geq V_k(\alpha) \geq \frac{c(k,\vep)}{[k^{\ab}(\alpha):k^{\ab}]^{2+\vep}}.
\end{equation}
Note that $c(k,\vep)$ depends only on $\Gamma$ and $\vep$, and $[k^\ab(\alpha):k^\ab] \leq [k^\cyc(\alpha):k^\cyc]$, so the above inequality gives (\ref{part1}).  

Now assume $\alpha \in k^\di$, so some power of $\alpha$ lies in $k$.  If $\alpha^m$ lies in $k$ for some $m \geq 1$, and then certainly $\alpha^m$ lies in $K_\Gamma$, and we proceed as follows.  First, we'll assume $m$ is the least integer such that $\alpha^m \in K_\Gamma$.  Then, by the Capelli-Kneser Theorem (see \cite[Chapter 2, Theorems 19 and 20]{schinzelpoly}, noting that $K_\Gamma$ contains all roots of unity), the polynomial $x^m-\alpha^m$ is irreducible over $K_\Gamma$, and so $m = [K_\Gamma(\alpha):K_\Gamma]$.  
Now, using the degree one form of Conjecture 1.1 and the fact that $\alpha^m \beta^m \in K_\Gamma\setminus \Gamma$ whenever $\beta\in \Gamma$, we have
\begin{align}
h(\alpha) \geq h_\Gamma(\alpha) = \inf_{\beta \in \Gamma} h(\alpha\beta) = \inf_{\beta \in \Gamma} \frac{1}{m}h(\alpha^m\beta^m)= \frac{\inf_{\beta \in \Gamma} h(\alpha^m\beta^m)}{[K_\Gamma(\alpha):K_\Gamma]} \geq \frac{C_\Gamma^0}{[K_\Gamma(\alpha):K_\Gamma]},
\end{align}
which gives the bound (\ref{sf}) expected in the strong form of Conjecture \ref{gl}.  This is much stronger than the conclusion of Conjecture \ref{rem2} (note that $k^\cyc \subseteq K_\Gamma$), and completes our proof that the degree one form of Conjecture \ref{gl} implies Conjecture \ref{rem2}.

\end{proof} 
 
\bibliography{gvb}{}
\today
\end{document}